\documentclass[11pt,a4paper,twoside]{amsart}
\usepackage{amsmath,amssymb,amsthm}
\usepackage{inputenc}
\usepackage{graphicx,subfigure}
\usepackage[T1]{fontenc} 
\usepackage{enumitem}
\usepackage{mathtools}
\usepackage{mathrsfs}
\usepackage{float}
\usepackage{bbm}

\usepackage[colorlinks=true, linkcolor=red, citecolor=blue, urlcolor=magenta,hyperfootnotes=false]{hyperref}
\usepackage{cleveref}
\usepackage{upgreek}

\def\@begintheorem#1#2{\par\bgroup{\sc #1 \ #2. }  \it \\\ignorespace }
\def\@opargbegintheorem#1#2#3{\par\bgroup{\sc #1\ #2 \ (#3).}  \it  \ignorespace}
\def\@endtheorem{\egroup}
\theoremstyle{plain}
\newtheorem{theorem}{Theorem}[section]

\theoremstyle{definition}

\theoremstyle{remark}
\newtheorem{remark}[theorem]{Remark}

\theoremstyle{plain}
\newtheorem{lemma}[theorem]{Lemma}

\theoremstyle{plain}
\newtheorem{corollary}[theorem]{Corollary}

\theoremstyle{plain}

\theoremstyle{plain}
\newtheorem{question}[theorem]{Question}

\newcommand{\C}{{\mathbb C}}

\newcommand{\BB}{{\mathcal{B}}}

\newcommand{\R}{{\mathbb R}}

\newcommand{\Vep}{\mathbf{V}_{\!\epsilon}}

\newcommand{\Rt}{{\R}^3}

\newcommand{\pO}{\partial\Omega}

\newcommand{\stwo}{\mathbb{S}^2}

\DeclarePairedDelimiter{\seq}{\lbrace}{\rbrace}

\newcommand{\hx}{\hat{x}}

\newcommand{\J}{\text{J}}
\newcommand{\Y}{\text{Y}}
\newcommand{\II}{\text{I}}
\newcommand{\KK}{\text{K}}

\begin{document}
\title[The relativistic spherical $\delta$-shell interaction in $\R^3$]{The relativistic spherical $\delta$-shell interaction in $\mathbb{R}^3$: spectrum and approximation}
\author{Albert Mas and Fabio Pizzichillo}
\address{A.\! Mas, Departament de Matem\`atiques i Inform\`atica,
Universitat de Barcelona. Gran Via Corts Catalanes 585,
08007 Barcelona (Spain).}
\email{albert.mas@ub.edu}
\address{F.\! Pizzichillo, BCAM - Basque Center for Applied Mathematics.
Alameda de Mazarredo 14, 48009 Bilbao (Spain).}
\email{fpizzichillo@bcamath.org}
\date{\today}

\begin{abstract}
This note revolves on the free Dirac operator in $\Rt$ and its $\delta$-shell interaction with electrostatic potentials supported on a sphere. On one hand, we characterize the eigenstates of those couplings by finding sharp constants and minimizers of some precise inequalities related to an uncertainty principle. On the other hand, we prove that the domains given by Dittrich, Exner and \v{S}eba 
{\em [Dirac operators with a spherically symmetric $\delta$-shell interaction, J. Math. Phys. 30.12 (1989), 2875-2882]} and by Arrizabalaga, Mas and Vega 
{\em [Shell interactions for Dirac operators, J. Math. Pures et Appl. 102.4 (2014), 617-639]} for the realization of an electrostatic  spherical shell interaction coincide. Finally, we explore the spectral relation between the shell interaction and its approximation by short range potentials with shrinking support, improving previous results in the spherical case.
\end{abstract}

\maketitle

\section{Introduction}
Since 1929, when Paul Dirac formulated it, the ``Dirac equation'' has played a fundamental role in quantum mechanics attracting the attentions of both mathematicians and physicists. 
The newness of Dirac's approach comes from the fact that he could describe the behaviour of the particles by using a $4$-component spinor. With this approach one can observe the existence of negative-energy particles, see \cite{thaller} and the references therein. 
By the way, it is not clear whether one should interpret the Dirac equation as a quantum mechanical evolution equation, like the Schr\"odinger equation for a single particle. 
An example that highlights this difference is, for instance, the comparison between the Schr\"odinger operator and the Dirac operator coupled with a singular potential supported on a hyper-surface: the so-called $\delta$ interaction.

The idea of coupling Hamiltonians with singular potentials supported on subsets of lower dimension with respect to the ambient space is quite classic in quantum mechanics.
For the Schr\"odinger operator, this problem is described in the monograph \cite{albeverio2012solvable} for finite and infinite $\delta$-point interactions and in \cite{exner2007leaky} for singular potentials supported on hyper-surfaces.
Regarding the Dirac operator, in the $1$-dimensional case the problem is well-understood. Thanks to \cite{albeverio2012solvable, posilicanoboundary, gesztesy1987new} we get the description of the domain, some properties of the spectrum and a resolvent formula. In three dimensions the first result is \cite{ditexnseb}. By using the decomposition into spherical harmonics, they can reduce their analysis to a $1$-dimensional question and they can construct the domain of the Dirac operator coupled with a singular potential supported on the sphere.
In the case of a general surface $\Sigma$, the first work is \cite{amv1}. In this work they characterize the domain of the $\delta$-shell Dirac operator with coupling constant $\lambda\neq\pm 2$, by the interactions between certain functions $u \in H^1(\Rt)$ and $g \in L^2(\Sigma)$.
Comparing this work with the general abstract theory given in \cite{posilicano1}, one could suppose that this kind of interaction is forcing $g$ to be in $H^{1/2}(\Sigma)$. Indeed, recently, in \cite{ourmieres2016strategy} they proved that this conjecture is true. Moreover they also define the domain of  $\delta$-shell Dirac operator when the coupling constant $\lambda = \pm 2$. Finally, in \cite{behrndt2016dirac, behrndt2016spectral} they can define the domain of the $\delta$-shell Dirac operator by using the abstract theory of boundary triples.

Nevertheless, one has to keep in mind that, even if this kind of model is easier to be mathematically understood, it is and ideal model that cannot be physically reproduced. This is the reason why it is interesting to approximate this kind of operators by more regular ones. And the results we obtain show the difference between the Schr\"odinger operator and the Dirac operator. For instance, in one dimension, if $V\in C^\infty_c(\R)$ then 
\begin{equation*}
V_\epsilon(t):=\textstyle{\frac{1}{\epsilon}\,V\big(\frac{t}{\epsilon}\big)
\to(\int V)}\delta_0\quad\text{when }\epsilon\to0
\end{equation*}
in the sense of distributions, where $\delta_0$ denotes the Dirac measure at the origin. 
In \cite{albeverio2012solvable} it is proved that  
$\Delta+V_\epsilon\to\Delta+(\int V)\delta_0$ in the norm resolvent sense when $\epsilon\to0$, and in \cite{approximation} this result is generalized to higher dimensions for singular perturbations on general smooth hyper-surfaces.

These kind of results do not hold for the Dirac operator and it is physically known as ``Klein's paradox''. In fact, in \cite{sebaklein} it is proved that, in the $1$-dimensional case,  the convergence holds in the norm resolvent sense but the coupling constant does depend non-linearly on the potential $V$, unlike in the case of Schr\"odinger operators.
Though, in the $3$-dimensional case, the convergence holds in the strong resolvent sense for bounded smooth hyper-surfaces under certain hypothesis over the potential $V$, see \cite{mp}. 

This note revolves on the free Dirac operator in $\Rt$ and its $\delta$-shell interactions with singular electrostatic potentials supported on a sphere. On one hand, we answer an open question posed in \cite{amv2} which provides eigenstates of those couplings by finding sharp constants and minimizers of some precise inequalities related to an uncertainty principle (see Question \ref{ques}, Theorem \ref{thm: positive answer} and Corollary \ref{coro: minimizers}). On the other hand, we prove that the domains given in \cite{ditexnseb} and \cite{amv1} coincide in the spherical case and that the conjecture that comes from the comparison to \cite{posilicano1} holds (see Theorem \ref{theo equal extensions} and Remark \ref{rem: H1/2}). Moreover we explore the spectral relation between the electrostatic $\delta$-shell interaction and its approximation by the coupling of the free Dirac operator with shrinking short range potentials, improving the results in \cite{mp} in the case of a spherical shell interaction (see Theorem \ref{thm: aprox-limit}).

\section{Preliminaries}
In this section we review some notation and basic rudiments on the construction of the electrostatic $\delta$-shell interactions mentioned in the introduction. We first recall the approach presented in \cite{amv1} for interactions on general smooth bounded domains and then we review the one from \cite{ditexnseb} in the case of a spherical interaction. At the end of the section we prove that both self-adjoint realizations coincide in the spherical case whenever the strength of the interaction differs from a concrete value. 

Given $m\geq0$, the free Dirac operator in $\R^3$ is defined by 
\begin{equation*}
H=-i\alpha\cdot\nabla+m\beta, 
\end{equation*}
where $\alpha=(\alpha_1,\alpha_2,\alpha_3)$,
\begin{equation*}
\alpha_i=\left(
\begin{array}{cc}
0&{\sigma}_i\\
{\sigma}_i&0
\end{array}
\right)
\quad\text{for }i=1,2,3,\quad 
\beta=\left(\begin{array}{cc}
\mathbb{I}_2&0\\
0&-\mathbb{I}_2
\end{array}\right),\quad
\mathbb{I}_2=\left(
\begin{array}{cc}
1 & 0\\
0 & 1
\end{array}\right),
\end{equation*}

\begin{equation*}
\text{and }\quad{\sigma}_1 =\left(
\begin{array}{cc}
0 & 1\\
1 & 0
\end{array}\right),\quad {\sigma}_2=\left(
\begin{array}{cc}
0 & -i\\
i & 0
\end{array}
\right),\quad{\sigma}_3=\left(
\begin{array}{cc}
1 & 0\\
0 & -1
\end{array}\right)
\end{equation*}
is the family of \textit{Pauli's matrices}. Thus $H$ acts on spinors of the form $\varphi:\Rt\to\C^4$. Despite the massless case has its own interest, throughout this article we will assume that $m>0$.
 
Let $\Omega\subset\Rt$ be a bounded smooth domain. We set $\Omega_+=\Omega$ and $\Omega_-=\Rt\setminus \overline{\Omega_+}$. Besides, let $\upsigma$ and $\nu$ denote the surface measure and unit outward (with respect to $\Omega$) normal vector field on $\partial\Omega$, respectively. A fundamental solution of $H$ is given by
\begin{equation*}
\phi (x)=\frac{e^{-m|x|}}{4\pi|x|}
\left(m\beta +\left(1+m|x|\right)\,i\alpha\cdot\frac{x}{|x|^2}\right)\quad \text{for }x\in\Rt\setminus\{0\},
\end{equation*}
see \cite[Lemma 3.1]{amv1}. Given $g\in L^2(\pO)^4$ we introduce the operator
\begin{equation}\label{defi Phia}
\Phi(g)(x)=
\int_{ {\pO}}\phi (x-y)g(y)\,d\upsigma(y)
\quad\text{for }x\in\Rt\setminus\pO.
\end{equation}
By \cite[Corollary 2.3]{amv1}, $\Phi: L^2(\pO)^4\to L^2(\Rt)^4$ is linear and bounded.
For $x\in\partial\Omega$ we also define
\begin{equation}\label{eq: Csigma def}
C_\upsigma g(x)
=\lim_{\epsilon\searrow 0}\int_{\pO\cap\{|x-y|>\epsilon\}}\phi (x-y) g(y)\,d\upsigma(y)
\quad\text{and}\quad C_{\pm}g(x)=\lim_{\Omega_\pm\ni y \overset{nt}{\to}x}\Phi(g)(y),
\end{equation}
where $\Omega_\pm\ni y \overset{nt}{\to}x$ means that $y$ tends to $x$ non-tangentially from $\Omega_{\pm}$, respectively. From \cite[Lemma 3.3]{amv1} we know that both $C_\upsigma$ and $C_\pm$ are linear and bounded in $L^2(\pO)^4$, and the following Plemelj-Sokhotski jump formulae holds:
\begin{equation}\label{Plemelj jump formulae}
C_\pm=\mp \frac{i}{2}(\alpha\cdot\nu)+C_\upsigma.
\end{equation}

Let $H^1(\Rt)^4$ denote the Sobolev space of $\C^4$-valued functions such that all its components have all its zero and first order derivatives in $L^2(\Rt)$. It is well known that the trace operator on $\pO$, initially defined on smooth functions, extends to a bounded operator 
\begin{equation*}
\operatorname{tr}_{\pO}:H^1(\Rt)^4\to H^{1/2}( \pO)^4 \subset L^2(\pO)^4,
\end{equation*}
where $H^{1/2}( \pO)$ denotes the Sobolev-Slobodeckij space on $\pO$ of order $1/2$. 

We are ready to introduce the electrostatic $\delta$-shell interaction $H+\lambda\delta_{\pO}$ studied in \cite{amv1}. Here, $\lambda\in\R$ is a paramenter that represents the strength of the interaction. Following 
\cite[Theorem 3.8]{amv1}, we define
\begin{equation}\label{eq defi electro}
\begin{split}
D(H+\lambda\delta_{\pO})
&=
\{u+\Phi(g): \,u\in H^1(\Rt)^4,\, g \in L^2(\pO)^4,\,
\lambda\operatorname{tr}_{\pO}u
=
-(1+\lambda C_\upsigma)g\},
\\
(H+\lambda\delta_{\pO})\varphi&=
H\varphi+\lambda\frac{\varphi_++\varphi_-}{2}\,\upsigma
\quad
\text{for }\varphi\in D(H+\lambda\delta_{\pO}),
\end{split}
\end{equation}
where $H\varphi$ in the right hand side of the second statement in \eqref{eq defi electro} is understood in the sense of distributions and $\varphi_\pm$ denotes the boundary traces of $\varphi$ when one approaches to $ {\pO}$ from $\Omega_\pm$. 
In particular, one has 
$(H+\lambda\delta_ {\pO})\varphi=H u \in L^2(\Rt)^4$
for all $\varphi=u+\Phi(g)\in D(H+\lambda\delta_ {\pO})$.
Furthermore, $H+\lambda\delta_ {\pO}$ is self-adjoint for all $\lambda\neq \pm 2$. For shortness sake, we put 
\begin{equation*}
H_\lambda=H+\lambda\delta_ {\pO}.
\end{equation*} 
Roughly speaking, to construct the interaction $H_\lambda$ presented in \cite{amv1}, one finds the domain of definiton of the operator looking in the space of distributions and imposing some restrictions to get a big enough domain for the adjoint. However, in \cite{ditexnseb}, the electrostatic $\delta$-shell interactions are obtained using self-adjoint extensions of the restricted operator $H|_{C_c^\infty(\Rt\setminus\pO)^4}$, as we will see below.

Unless we say the contrary, from now on we restrict our study to the case 
\begin{equation*}
\Omega=\{x\in\R^3:\,|x|<1\}.
\end{equation*}
For clarity, let us denote $B_\pm=\Omega_\pm$ and $\stwo=\pO$. We now review the approach from  \cite{ditexnseb}, where the authors construct self-adjoint and rotationally invariant extensions of $H|_{C_c^\infty(\Rt\setminus\stwo)^4}$ by using the decomposition in the classical spherical harmonics $Y^l_n$. The indices refer to  $n = 0, 1, 2, \dots$ and $l =-n,-n + 1,\dots , n,$, and the functions satisfy $\Delta_{\stwo} Y^l_n= n(n + 1)Y^l_n$, where $\Delta_{\mathbb{S}^2}$ denotes the usual spherical laplacian. Moreover, $\{Y^l_n\}_{l,n}$ is a complete orthonormal set in $L^2(\mathbb{S}^2)$.
For $j = 1/2, 3/2, 5/2, \dots  $ and $m_j = -j,-j + 1, \dots , j$ set
\begin{align*}
\psi^{m_j}_{j-1/2}&=
\frac{1}{\sqrt{2j}}
\left(\begin{array}{c}
\sqrt{j+m_j}\,Y^{m_j-1/2}_{j-1/2}\\
\sqrt{j-m_j}\,Y^{m_j+1/2}_{j-1/2}\\
\end{array}\right),
\\
\psi^{m_j}_{j+1/2}&=\frac{1}{\sqrt{2j+2}}
\left(\begin{array}{c}
\sqrt{j+1-m_j}\,Y^{m_j-1/2}_{j+1/2}\\
-\sqrt{j+1+m_j}\,Y^{m_j+1/2}_{j+1/2}\\
\end{array}\right).
\end{align*}
Then $\{\psi^{m_j}_{j\pm 1/2}\}_{j,m_j}$ is a complete orthonormal set in $L^2(\mathbb{S}^2)^2$. Moreover, if we set 
\begin{equation*}
r=|x|,\quad\hat x = x / |x|\quad\text{and}\quad
L=-ix\times\nabla\quad\text{for }x \in \Rt\setminus\{0\},
\end{equation*}
then
\begin{equation*}
(\sigma\cdot\hx)\psi^{m_j}_{j\pm 1/2}=\psi^{m_j}_{j\mp1/2},
\quad\text{and}\quad
(1+\sigma\cdot L)\psi^{m_j}_{j\pm 1/2}=\pm(j+1/2)\psi^{m_j}_{j\pm 1/2},
\end{equation*}
where $\sigma=(\sigma_1,\sigma_2,\sigma_3)$ is the vector of \textit{Pauli's matrices}.
For $k_j=\pm(j+1/2)$ we define
\begin{equation*}
\Phi^+_{m_j,k_j}=
 \left(\begin{array}{c}
i\,\psi^{m_j}_{j\pm 1/2}\\
0
\end{array}\right)
\quad\text{and}\quad
\Phi^-_{m_j,k_j}=
\left(\begin{array}{c}
0\\
\psi^{m_j}_{j\mp1/2}
\end{array}\right).
\end{equation*}
Then, the set $\BB=\{\Phi^+_{m_j,k_j},\Phi^-_{m_j,k_j}\}_{j,k_j,m_j}$ is a 
complete orthonormal base of $L^2(\mathbb{S}^2)^4$. We refer to 
\cite[{Section 4.6}]{thaller} for the details. 

We now introduce the subspaces
\begin{align*}
\displaystyle
\mathcal{C}_{m_j,k_j}
&=\seq*{
\frac{1}{r}\left(
f^+_{m_j,k_j}(r)\Phi^+_{m_j,k_j}(\hx)+
f^-_{m_j,k_j}(r)\Phi^-_{m_j,k_j}(\hx)\right)
\in L^2(\Rt)^4
:\,
f^\pm_{m_j,k_j}\in C^\infty_c(0,+\infty)}.
\end{align*}
From \cite[Theorems 1.1 and 4.14]{thaller} we know that the operator 
\begin{equation*}
\mathring H=H|_{C^\infty_c(\Rt\setminus\seq{0})^4}
\end{equation*}
is essentially self-adjoint and
leaves the partial wave subspace $\mathcal{C}_{m_j,k_j}$ invariant. Its action on each subspace is represented, in terms of the base $\BB$, by the operator  
\begin{equation}\label{eq:dirac.spherical}
{D}( \mathring  h_{m_j,k_j})= C^\infty_c(0,+\infty)^2,
\quad
\mathring h_{m_j,k_j}
\begin{pmatrix}
  f \\
  g
\end{pmatrix}
=\left(
\begin{array}{cc}
m & -\frac{d}{dr}+\frac{k_j}{r}\\
\frac{d}{dr}+\frac{k_j}{r} & -m
\end{array}\right)
\begin{pmatrix}
  f \\
 g
\end{pmatrix}.
\end{equation}
By \cite[Lemma 4.15]{thaller}, $\mathring h_{m_j,k_j}$ is essentially self-adjoint and, if we set $h_{m_j,k_j}=\overline{\mathring h_{m_j,k_j}}$, we get that 
\begin{align*}
%\mathring H&\cong
%\bigoplus_{j=\frac{1}{2},\frac{3}{2},\dots}^\infty\,\bigoplus_{m_j=-j}^j\,
%\bigoplus_{k_j=\pm(j+1/2)}\,
%\mathring{h}_{m_j,k_j}, \text{\textcolor{red}{ (is this line correct? they are $C^\infty_c$ functions, as in what we removed)}}
%\\
 H&\cong
\bigoplus_{j=\frac{1}{2},\frac{3}{2},\dots}^\infty\,\bigoplus_{m_j=-j}^j\,
\bigoplus_{k_j=\pm(j+1/2)}\,
{h}_{m_j,k_j},
\end{align*}
where ``$\cong$'' means that the operators are unitarily equivalent 
and $H$ is the free Dirac operator defined on $H^1(\R^3)$.

For $m_j$ and $k_j$ as above, let us define 
\begin{equation*}
{D}(\hat h_{m_j,k_j})=C^\infty_c\left((0,1)\cup(1,+\infty)\right)^2
\subset {D}(\mathring h_{m_j,k_j}),
\quad
\hat h_{m_j,k_j}\varphi:=\mathring h_{m_j,k_j}\varphi,\
\text{for all}\ \varphi\in{D}(\hat h_{m_j,k_j}).
\end{equation*}
Assume that $\lambda\in\R\setminus\{\pm 2\}$ and set
\begin{equation*}
M^\pm_\lambda=\left( \begin{array}{cc}
\lambda/2&\pm 1\\
\mp 1 & \lambda / 2 
\end{array}\right).
\end{equation*}
In \cite{ditexnseb} it is proved that
the operator ${h}(\lambda)_{m_j,k_j}$ defined by
\begin{equation}\label{def: h(lambda)mj,kj}
\begin{split}
{D}({h}(\lambda)_{m_j,k_j})
&=
\Big\{
\varphi\in L^2(0,+\infty)^2: h_{m_j,k_j}\varphi\in L^2(0,+\infty)^2,\, \varphi\in AC\big((0,1)\cup(1,+\infty)\big)^2,\\
&\hskip105pt M^-_\lambda \varphi(1^+)+M^+_\lambda \varphi(1^-)=0
\Big\},\\
{h}(\lambda)_{m_j,k_j}\varphi&={h}_{m_j,k_j}\varphi\quad\text{for all }\varphi\in{D}({h}(\lambda)_{m_j,k_j})
\end{split}
\end{equation}
is a self-adjoint extension of $\hat h_{m_j,k_j}$. Here, $AC\big((0,1)\cup(1,+\infty)\big)$ denotes the space of absolutely continuous functions on the open set $(0,1)\cup(1,+\infty)$. Furthermore, if one sets 
\begin{equation*}
\delta_1(\varphi)=\frac{\varphi(1^+)+\varphi(1^-)}{2},
\end{equation*}
then ${h}(\lambda)_{m_j,k_j}=\mathring  h_{m_j,k_j}+\lambda\delta_1$ on 
$D({h}(\lambda)_{m_j,k_j})$, with the understanding that here $\mathring  h_{m_j,k_j}$ just means the differential operator given by the matrix on the right hand side of \eqref{eq:dirac.spherical} acting in the sense of distributions.
Let us finally introduce the subspaces
\begin{equation*}
\mathcal{H}(\lambda)_{m_j,k_j}=
\seq*{ 
\frac{1}{r}\Big(
f^+_{m_j,k_j}(r)\Phi^+_{m_j,k_j}(\hx)+
f^-_{m_j,k_j}(r)\Phi^-_{m_j,k_j}(\hx)\Big)\!
\in L^2(\Rt)^4:\,
f^\pm_{m_j,k_j}\!\!\in {D}({h}(\lambda)_{m_j,k_j})}. 
\end{equation*}
The electrostatic $\delta$-shell interaction with strength $\lambda$ studied in \cite{ditexnseb} is given by 
\begin{equation}\label{eq:def.delta.shell.dittexn}
\begin{split}
{D}(\widehat{H}(\lambda))&=
\bigoplus_{j=\frac{1}{2},\frac{3}{2},\dots}^\infty\,\bigoplus_{m_j=-j}^j\,
\bigoplus_{k_j=\pm(j+1/2)}\,
\mathcal{H}(\lambda)_{m_j,k_j},
\\
\widehat{H}(\lambda)
&\cong
\bigoplus_{j=\frac{1}{2},\frac{3}{2},\dots}^\infty\,\bigoplus_{m_j=-j}^j\,
\bigoplus_{k_j=\pm(j+1/2)}\,
h(\lambda)_{m_j,k_j},
\end{split}
\end{equation}
which is a self-adjoint operator.

In order to compare the notions of a $\delta$-shell interaction given by \eqref{eq defi electro} in the spherical case and 
\eqref{eq:def.delta.shell.dittexn}, let us first prove an auxiliary result.

\begin{lemma}\label{lem:f=u+phi(x)g}
Let $\Omega\subset \Rt$ be a bounded domain of class $C^2$. Then
\begin{equation*}
H^1(\R^3\setminus {\pO})^4=\seq{u+\Phi(g): u \in H^1(\R^3)^4,\, g\in H^{1/2}( \pO)^4}.
\end{equation*}
\end{lemma}
\begin{proof}
If $f=u+\Phi(g)$ for some $u \in H^1(\Rt)^4$ and $g\in H^{1/2}( {\pO})^4$, by \cite[Lemma 3.1]{posilicanoboundary} we have that $f \in H^1(\R^3\setminus  {\pO})^4$. 

Let us consider now $f \in  H^1(\R^3\setminus {\pO})^4$. 
Since $f \in  H^1(\Omega_\pm)^4$, by the trace theorem we also have 
$f_\pm\in H^{1/2}( {\pO})^4$. Set
\begin{equation*}
g:= i\,(\alpha\cdot \nu)(f_+-f_-)\in H^{1/2}(\pO)^4
\end{equation*}
and $u=f-\Phi(g)$. Once again, \cite[Lemma 3.1]{posilicanoboundary} shows that $u \in   H^1(\R^3\setminus {\pO})^4$. Moreover, by \eqref{Plemelj jump formulae},
\begin{equation*}
u_+-u_-=f_+-f_--C_+g+C_-g
=f_+-f_-+i\,(\alpha\cdot \nu)g=0,
\end{equation*}
thus $u_+=u_-$ and $u$ has a well defined boundary trace in $H^{1/2}( {\pO})^4$. This implies that actually $u\in H^1(\Rt)^4$, and we are done since $f=u+\Phi(g)$.
\end{proof}

\begin{theorem}\label{theo equal extensions}
Assume that $\Omega=\{x\in\R^3:\,|x|<1\}$.
For any $\lambda \in\R\setminus\{\pm 2\}$, the self-adjoint realizations $H_\lambda$ and $\widehat{H}(\lambda)$ defined by \eqref{eq defi electro} and \eqref{eq:def.delta.shell.dittexn}, respectively, coincide.
\end{theorem}
\begin{proof}
Consider the operator
\begin{equation*}
\begin{split}
&D(\widetilde{H}_\lambda)=\seq{
u+\Phi(g):
u\in H^1(\R^3)^4,\,g\in H^{1/2}(\stwo)^4,\, 
\lambda \operatorname {tr}_{\stwo} u =-(1+\lambda C_{\sigma})g},\\
&\widetilde{H}_\lambda={H_\lambda}|_{{D(\widetilde{H}_\lambda)}}.
\end{split}
\end{equation*}
Since $ H^{1/2}(\stwo)^4\subset L^2(\stwo)^4$, by construction we get $\widetilde{H}_\lambda\subset H_\lambda$. We are going to prove that 
\begin{equation}\label{eq:incl.hatH.tildeH}
\widehat{H}(\lambda) \subset \widetilde{H}_\lambda.
\end{equation}
With this at hand, we deduce that $\widehat{H}(\lambda) \subset \widetilde{H}_\lambda\subset H_\lambda$ and, since both $\widehat{H}(\lambda)$ and $H_\lambda$ are self-adjoint operators for $\lambda\neq \pm 2$, we finally conclude that $\hat{H}(\lambda)=H_\lambda$ and the theorem follows. Let us focus on \eqref{eq:incl.hatH.tildeH}.
Fixed $m_j$ and $k_j$ as in \eqref{eq:def.delta.shell.dittexn}, for simplicity of notation we put
\begin{align*}
f^\pm (r) &=
f^\pm_{j,m_j}(r),\\
\Phi^\pm(\hx)&=
\Phi^\pm_{j,m_j}(\hx),\\
\mathcal{H}(\lambda)&=\mathcal{H}(\lambda)_{m_j,k_j},\\
h(\lambda)&=h(\lambda)_{m_j,k_j}.
\end{align*}
Thus, any $\varphi\in\mathcal{H}(\lambda)$ can be written as 
\begin{equation*}
\varphi(x)=\frac{1}{r}
\left(f^+(r)\Phi^+(\hx)+f^-(r)\Phi^-(\hx)
\right)=
\frac{1}{r}
\begin{pmatrix}
f^+(r)\\
f^-(r)
\end{pmatrix}
\cdot
\begin{pmatrix}
\Phi^+(\hx)\\
\Phi^-(\hx)
\end{pmatrix}.
\end{equation*}
In the last expresion above, ``$\cdot$'' just means ``scalar product''. As before, we denote by $\varphi_\pm$ the boundary values of $\varphi$ when we approach $\stwo$ from $\Omega_\pm$.
Let $\mathcal{M}_\lambda^\pm$ be the operator defined on $\mathcal{H}(\lambda)$ by the action of the matrix $M^\pm_\lambda$ on the basis $\seq{\Phi^+,\Phi^-}$, that is, for any $\hx\in {\stwo}$, 
\begin{align*}
\mathcal{M}_\lambda^+ \varphi_+ (\hx)&=
\frac{1}{r}
\left(
M^+_\lambda\!
\begin{pmatrix}
f^+(1^+)\\
f^-(1^+)
\end{pmatrix}
\right)
\cdot
\begin{pmatrix}
\Phi^+(\hx)\\
\Phi^-(\hx)
\end{pmatrix},\\
\mathcal{M}_\lambda^- \varphi_- (\hx)&=
\frac{1}{r}
\left(
M^-_\lambda\!
\begin{pmatrix}
f^+(1^-)\\
f^-(1^-)
\end{pmatrix}
\right)
\cdot
\begin{pmatrix}
\Phi^+(\hx)\\
\Phi^-(\hx)
\end{pmatrix}.
\end{align*}
So, in particular, we have that  
\begin{equation}\label{eq:cond.C_lambda^pm}
\mathcal{M}^+_\lambda\varphi_-(\hx)+\mathcal{M}^-_\lambda\varphi_+(\hx)=0\quad\text{for all }\hx\in {\stwo}.
\end{equation}
Moreover, since $\varphi\in H^1(\Rt\setminus {\stwo})^4$, using \Cref{lem:f=u+phi(x)g} we can write $\varphi=u+\Phi (g)$ for some $u \in H^1(\R^3)^4$ and $g \in L^2(\stwo )^4$. Then, since $\nu(\hx)=\hx$ for all $\hx\in\stwo$, using \eqref{Plemelj jump formulae} we see that \eqref{eq:cond.C_lambda^pm} is equivalent to 
\begin{equation}\label{eq:cond.dittexn}
\begin{split}
0&=
(\mathcal{M}^+_{\lambda} +\mathcal{M}^-_\lambda)
\operatorname {tr}_{\stwo} u(\hx)
+ \big(\mathcal{M}^+_{\lambda} C_+ 
+ \mathcal{M}^-_\lambda C_- \big)g (\hx)
\\
&=
(\mathcal{M}^+_{\lambda} +\mathcal{M}^-_\lambda)
\operatorname {tr}_{\stwo} u(\hx)
+\frac{1}{2}\left(\mathcal{M}^-_\lambda -\mathcal{M}^+_\lambda\right) i(\alpha\cdot \hx)g(\hx)
+(\mathcal{M}^+_{\lambda} +\mathcal{M}^-_\lambda) C_{\upsigma} g(\hx).
\end{split}
\end{equation}
Since $M_\lambda^++M_\lambda^-=\lambda \mathbb{I}_2$, where $\mathbb{I}_2$ denotes the $2\times2$ identity matrix, we get that, for $\hx\in\stwo$, 
\begin{align}\label{eq:C++C-u}
(\mathcal{M}^+_{\lambda} +\mathcal{M}^-_\lambda) u(\hx)
&=
\lambda u(\hx),
\\
\label{eq:C++C-C_sigmag}
(\mathcal{M}^+_\lambda +\mathcal{M}^-_\lambda) C_{\upsigma} g(\hx)
&=
\lambda C_{\upsigma} g(\hx).
\end{align}
Note also that 
\begin{equation*}
\frac{1}{2}\left({M}^-_\lambda -{M}^+_\lambda\right)=
\begin{pmatrix}
0&-1\\
1&0
\end{pmatrix},
\end{equation*} 
that is the matrix that represent the operator $-i(\alpha \cdot \hx)$ on the basis $\seq{\Phi^+,\Phi^-}$ (see \cite[Equation 4.123]{thaller}). So 
\begin{equation}\label{eq:C--C+g}
\frac{1}{2}\left(\mathcal{M}^-_\lambda -\mathcal{M}^+_\lambda\right) (i\alpha\cdot \hx)g(\hx)=g(\hx)
\end{equation}
for $\hx\in\stwo$.
Combining \eqref{eq:C++C-u}, \eqref{eq:C++C-C_sigmag} and \eqref{eq:C--C+g}, \eqref{eq:cond.dittexn} becomes 
\begin{equation*}
0=\lambda \operatorname{tr}_{\stwo} u + (1+\lambda C_\upsigma)g.
\end{equation*}
In conclusion, we have seen that if $\varphi\in\mathcal{H}(\lambda)$ then $\varphi\in{H}_\lambda$. Since these arguments are valid for any $m_j$ and $k_j$, \eqref{eq:incl.hatH.tildeH} follows.
\end{proof}
\begin{remark}\label{rem: H1/2}
From the proof of \Cref{theo equal extensions} we also see that if $\lambda\neq \pm 2$ then $\widetilde{H}_\lambda=H_\lambda$, which means that the condition $\lambda \operatorname{tr}_{\stwo} u=-(1+\lambda C_\sigma)g$ in \eqref{eq defi electro} forces $g$ to belong to $H^{1/2}(\mathbb{S}^2)^4$.
\end{remark}

\section{On the spectrum for the spherical $\delta$-shell interaction}\label{sec:delta}
In this section we answer affirmatively a question posed in \cite[Section 4.2.3]{amv2}. As commented there, this yields a relation between  the eigenvalues in the gap $(-m,m)$ for the electrostatic spherical $\delta$-shell interaction and the minimizers of some precise quadratic form inequality. Before going further, we must recall some rudiments from \cite[Section 4]{amv2}.
Throughout this section, $\Omega$ denotes the unit ball and $\pO=\stwo$. Given $a\in[-m,m]$, set
\begin{equation*}
k^a(x)=\frac{e^{-\sqrt{m^2-a^2}|x|}}{4\pi|x|}\,\mathbb{I}_2
\quad\text{and}\quad
w^a(x)=\frac{e^{-\sqrt{m^2-a^2}|x|}}{4\pi|x|^3}
\left(1+\sqrt{m^2-a^2}|x|\right)\,i\,{\sigma}\cdot x
\end{equation*}
for $x\in\R^3\setminus\{0\}$. Given $f\in L^2(\upsigma)^2$ and $x\in \stwo$, set
\begin{equation*}
K^af(x)=\int_{\stwo} k^a(x-z)f(z)\,d\upsigma(z) 
\quad\text{and}\quad
W^af(x)=\lim_{\epsilon\searrow0}\int_{\{|x-z|>\epsilon\}\cap\stwo} w^a(x-z)f(z)\,d\upsigma(z).
\end{equation*}
Then 
\begin{equation*}
C_\upsigma^a
=\left(\begin{array}{cc}  (a+m)K^a& W^a\\
W^a & (a-m)K^a \end{array}\right),
\end{equation*}
where $C_\upsigma^a$ is defined as $C_\upsigma$ in \eqref{eq: Csigma def} replacing the kernel $\phi$ by 
\begin{equation*}
\phi^a(x)=\frac{e^{-\sqrt{m^2-a^2}|x|}}{4\pi|x|}\left(a+m\beta
+\left(1+\sqrt{m^2-a^2}|x|\right)\,i\alpha\cdot\frac{x}{|x|^2}\right)\quad\text{for }x\in\R^3\setminus\{0\}.
\end{equation*}
The following corresponds to \cite[Lemma 4.3]{amv2}.
\begin{lemma}\label{sphere l3}
Given $a\in(-m,m)$, there exist positive numbers $d_{j\pm 1/2}$ and purely imaginary numbers $p_{j\pm 1/2}$ for all $j=1/2,3/2,5/2,\ldots,$ and $m_j=-j,-j+1,\ldots,j$, such that
\begin{itemize}
\item[$(i)$] $K^a\,\psi_{j\pm 1/2}^{m_j}=d_{j\pm 1/2}\,\psi_{j\pm 1/2}^{m_j}$ and\,
$\lim_{j\to\infty}d_{j\pm 1/2}=0$. Moreover, 
$$0\leq d_{j\pm 1/2}\leq d_0=\frac{1-e^{-2\sqrt{m^2-a^2}}}{2\sqrt{m^2-a^2}}.$$
\item[$(ii)$] $W^a\,\psi_{j\pm 1/2}^{m_j}=p_{j\pm 1/2}\,\psi_{j\mp1/2}^{m_j}$ and \,$p_{j+1/2}=-p_{j-1/2}$. Moreover,
$$|p_{j\pm 1/2}|^2=\frac{1}{4}-(m^2-a^2)d_{j+1/2}\,d_{j-1/2}\geq
\frac{1}{4}\,e^{-2\sqrt{m^2-a^2}}\left(2-e^{-2\sqrt{m^2-a^2}}\right).$$
\end{itemize}
\end{lemma}
The following result allows us to construct eigenstates for $H_\lambda$ from the eigenfunctions of $K^a$; it corresponds to \cite[Lemma 4.6]{amv2}.
\begin{lemma}\label{lemma: criteria eigenstates}
Let $H_\lambda$ be as in \eqref{eq defi electro}. If $\lambda>0$ and $a\in (-m,m)$ satisfy
\begin{equation}\label{eq eigenvalue}
\frac{\lambda^2}{4}-\big((m+a)d_{j\mp 1/2}-(m-a)d_{j\pm 1/2}\big)\lambda=1\quad\text{for some } j,
\end{equation}
then, for any $m_j$, $\psi_{j\pm 1/2}^{m_j}$ gives rise to an eigenfunction for $H_\lambda$ with eigenvalue $a$.
\end{lemma}
\begin{remark}
In Lemma \ref{lemma: criteria eigenstates}, the expression ``gives rise to an eigenfunction'' means that, if one defines
\begin{equation*}
g=\begin{pmatrix}f\\h\end{pmatrix}\in L^2(\stwo)^4,\quad\text{where } h=\psi_{j\pm 1/2}^{m_{j}}\quad\text{and}\quad
f=-\big(1/\lambda+(a+m)K^a\big)^{-1}W^a\,h,
\end{equation*}
setting $\varphi=\phi*(a\Phi^a(g))+\Phi(g)$ one gets that $H_{\lambda}\varphi=a\varphi$. Here, $\Phi^a$ is defined as $\Phi$ in \eqref{defi Phia} replacing $\phi$ by $\phi^a$.
\end{remark}

In \cite[Question 4.7]{amv2}, the following question was raised:
\begin{question}\label{ques}
Let $d_{j\pm 1/2}$ be the coefficients given by {\em Lemma \ref{sphere l3}}. Is it true that $d_{j+1/2}d_{j-1/2}<d_1d_0$ for all $j=3/2, 5/2, 7/2\ldots$?
\end{question}
Theorem \ref{thm: positive answer} answers it in the affirmative and, as commented at the end of \cite[Section 4.2.3]{amv2}, it yields the following result related to Lemma \ref{lemma: criteria eigenstates}. We first recall the values of $d_0$ and $d_1$ from Lemma \ref{sphere l3} (computed in \cite{amv2}) and a precise constant $d_*$ that will appear below, see \cite[equations $(4.31)$, $(4.32)$ and $(4.39)$, respectively]{amv2}:
\begin{equation*}
\begin{split}
d_0&=\frac{1-e^{-2\sqrt{m^2-a^2}}}{2\sqrt{m^2-a^2}},\\
d_1&=\frac{1}{2\sqrt{m^2-a^2}}\bigg(1-\frac{1}{m^2-a^2}+\left(1+\frac{1}{\sqrt{m^2-a^2}}\right)^2e^{-2\sqrt{m^2-a^2}}\bigg),\\
d_*&=\frac{1}{2\sqrt{m^2-a^2}}
-\frac{1}{2}\left(1+\frac{1}{\sqrt{m^2-a^2}}\right)
e^{-2\sqrt{m^2-a^2}}.
\end{split}
\end{equation*}

\begin{corollary}\label{coro: minimizers}
Let $a\in(-m,m)$ and $\lambda>0$.
Then, for any $f\in L^2(\upsigma)^2$, 
\begin{equation}\label{ineqbis}
\begin{split}
\int_{\stwo} | f|^2\,d\upsigma
&\leq \frac{1/\lambda +(m+a)d_0}{2d_*^2}\int_{\stwo} \big(1/\lambda+(m+a)K^a\big)^{-1}
W^a f\cdot\overline{W^af}\,d\upsigma\\ 
&\quad+\frac{1}{2(1/\lambda +(m+a)d_0)}\int_{\stwo} \big(1/\lambda+(m+a)K^a\big)(\sigma\cdot \nu)f
\cdot\overline{(\sigma\cdot \nu) f}\,d\upsigma.
\end{split}
\end{equation}
The equality in $(\ref{ineqbis})$ is only attained at linear combinations of $\psi_1^{l}$ for $l\in\{-1/2,1/2\}$. If
\begin{equation}\label{a2}
\frac{\lambda^2}{4}-\big((m+a)d_{0}-(m-a)d_{1}\big)\lambda=1
\end{equation}
then the minimizers of $(\ref{ineqbis})$ give rise to eigenfunctions of $H_\lambda$. 
Besides, these conclusions also hold if we exchange the roles of $d_0$ and $d_1$ in $(\ref{ineqbis})$ and $(\ref{a2})$ and we replace $\psi_1^{l}$ by $\psi_0^{l}$ (that is, we exchange the roles of $j+1/2$ and $j-1/2$ for $j=1/2$).
\end{corollary}

\begin{theorem}\label{thm: positive answer}
Let $d_{j\pm 1/2}$ be the coefficients given by {\em Lemma \ref{sphere l3}}. Then,  
\begin{equation}\label{def dj}
d_{j\pm 1/2}=\operatorname{I}_{(j+1/2)\pm 1/2}\big(\sqrt{m^2-a^2}\big)\operatorname{K}_{(j+1/2)\pm 1/2}\big(\sqrt{m^2-a^2}\big),
\end{equation} 
where $\operatorname{I}$ and $\operatorname{K}$ denote the standard second order Bessel's functions.
Moreover,
\begin{equation}\label{conjecture}
d_{j+1/2}d_{j-1/2}<d_0 d_1\quad\text{ for all $j=3/2, 5/2, 7/2\ldots$.}
\end{equation}
\end{theorem}
\begin{proof}
Let us first compute $d_{j\pm 1/2}$ in terms of Bessel's functions. Fixed $m_j$ and $k_j$, due to \cite[Lemma 4.15]{thaller} and \Cref{theo equal extensions} it is enough to find some $a\in (-m,m)$ which is an eigenvalue for the operator $h(\lambda)_{m_j,k_j}$. We want to find some
\begin{equation*}
\begin{pmatrix}f\\g\end{pmatrix}\in {D}(h(\lambda)_{m_j,k_j})
\end{equation*}
verifying the following system of differential equations:
\begin{equation}\label{differential system}
\begin{cases}
(m-a)f+(-\frac{d}{dr}+\frac{{k_j}}{r})g&=0,\\
(\frac{d}{dr}+\frac{{k_j}}{r})f-(m+a)g&=0.
\end{cases}
\end{equation}
Set $M=\sqrt{m^2-a^2}$. Since $k_j=\pm (j+1/2)$, we set
\begin{equation}\label{eq: definition phi k_j>0}
f(r)=\!\begin{cases}
A\sqrt{r}\,\II_{(j+1/2)\pm 1/2}(Mr)&\!\!\text{if }r<1\\
B\sqrt{r}\,\KK_{(j+1/2)\pm 1/2}(Mr)&\!\!\text{if } r>1
\end{cases}\!,\quad \!
g(r)=\!\begin{cases}
\frac{ AM}{m+a}\sqrt{r}\,\II_{(j+1/2)\mp1/2}(Mr)&\!\!\text{if }r<1\\
-\frac{B M}{m+a}\sqrt{r}\,\KK_{(j+1/2)\mp1/2}(Mr)&\!\!\text{if } r>1
\end{cases}\!,
\end{equation}
for some $(A,B)\neq (0,0)$. If we put 
\begin{equation*}
\varphi=\begin{pmatrix}f\\g\end{pmatrix},
\end{equation*}
then $\varphi \in  L^2(0,+\infty)^2$, 
$h_{m_j,k_j}\varphi\in L^2(0,+\infty)^2$, 
$\varphi\in AC\big((0,1)\cup(1,+\infty)\big)^2$ and $\varphi$ satisfies \eqref{differential system}. Thus, to get that $\varphi$ is an eigenvector for the operator $h(\lambda)_{m_j,k_j}$ it remains to prove that $\varphi\in D(h(\lambda)_{m_j,k_j})$, that is we have to show that $ M^-_\lambda \varphi(1^+)+M^+_\lambda \varphi(1^-)=0$. In other words, the following linear system must hold:
\begin{equation*}
\begin{cases}
\begin{array}{l}
 A \left(\frac{M}{a+m} \II_{(j+1/2)\mp 1/2}(M)+\frac{\lambda}{2}  \II_{(j+1/2)\pm 1/2}(M)\right)\\
\quad +B \left(\frac{M }{a+m}\KK_{(j+1/2)\mp 1/2}(M)+\frac{\lambda}{2}  \KK_{(j+1/2)\pm 1/2}(M)\right) = 0,\\
\\
 A \left(\frac{\lambda M }{2 (a+m)}\II_{(j+1/2)\mp 1/2}(M)- \II_{(j+1/2)\pm 1/2}(M)\right)\\
\quad +B \left(\KK_{(j+1/2)\pm 1/2}(M)-\frac{\lambda M }{2 (a+m)}\KK_{(j+1/2)\mp 1/2}(M)\right)=0.
\end{array}
\end{cases}
\end{equation*}
Since this is a $2\times 2$ homogeneous linear system on $A$ and $B$ and we are supposing that $(A,B)\neq (0,0)$, we deduce that the associated matrix has null determinant. This means that
\begin{equation}\label{eq:determinant.delta}
\begin{split}
0&= -\frac{\lambda ^2M }{4(a+m)}\left(\II_{(j+1/2)\pm 1/2}(M)
 \KK_{(j+1/2)\mp 1/2}(M)
 +
 \II_{(j+1/2)\mp 1/2}(M) \KK_{(j+1/2)\pm 1/2}(M)
\right) 
   \\
&\quad
+
\frac{\lambda}{m+a}
\left(
(m+a)
\II_{(j+1/2)\pm 1/2}(M) \KK_{(j+1/2)\pm 1/2}(M)
-
(m-a)
\II_{(j+1/2)\mp 1/2}(M) \KK_{(j+1/2)\mp 1/2}(M)
\right)
\\
&\quad+
\frac{M}{m+a}
\left(
 \II_{(j+1/2)\pm 1/2}(M) \KK_{(j+1/2)\mp 1/2}(M)
+
\II_{(j+1/2)\mp 1/2}(M) \KK_{(j+1/2)\pm 1/2}(M)
\right)
.
\end{split}
\end{equation}
By \cite[Equation 10.20.2]{olver2010nist} we get that 
\begin{equation}\label{eq:wronskian}
\II_{(j+1/2)\pm 1/2}(M)
 \KK_{(j+1/2)\mp 1/2}(M)
 +
 \II_{(j+1/2)\mp 1/2}(M) \KK_{(j+1/2)\pm 1/2}(M)=
 \frac{1}{M}.
\end{equation}
Finally, combining \eqref{eq:determinant.delta} and \eqref{eq:wronskian} we see that the following must hold:
\begin{equation}\label{eq:final.eigenvalue}
\begin{split}
\frac{\lambda^2}{4}-\Big((m+a)&\II_{(j+1/2)\pm 1/2}(M)\KK_{(j+1/2)\pm 1/2}(M)\\
&-(m-a)\II_{(j+1/2)\mp 1/2}(M)\KK_{(j+1/2)\mp 1/2}(M)\Big)\lambda-1=0.
\end{split}
\end{equation}
In conclusion, if we define
\begin{equation}\label{condition limit}
\begin{split}
D_{j\pm1/2}(a,\lambda)&=
\frac{\lambda^2}{4}-\Big((m+a)\II_{(j+1/2)\pm 1/2}(M)\KK_{(j+1/2)\pm 1/2}(M)\\
&\quad-(m-a)\II_{(j+1/2)\mp 1/2}(M)\KK_{(j+1/2)\mp 1/2}(M)\Big)\lambda-1
\end{split}
\end{equation}
and we take 
$\varphi=\begin{pmatrix}f\\g\end{pmatrix}$ with $f$ and $g$ given by
\eqref{eq: definition phi k_j>0}, 
then $\varphi$ is an eigenfunction for $h(\lambda)_{m_j,k_j}$ with eigenvalue $a$ if and only if $D_{j\pm1/2}(a,\lambda)=0$. In this case 
the function 
\begin{equation*}
\psi(x)=
\frac{1}{r}
\left(
f(r)\Phi^+_{m_j,k_j}(\hx)
+g(r)\Phi^-_{m_j,k_j}(\hx)
\right)
\end{equation*}
is an eigenfunction for $H_\lambda$ with eigenvalue $a$.
For this reason, a comparison of \eqref{eq:final.eigenvalue} and \eqref{eq eigenvalue} yields  \eqref{def dj}, as desired.

Let us finally prove \eqref{conjecture}. We put $n=j+1/2\in \mathbb{N}$. Since $j>1/2$, we have $n>1$. Then \eqref{conjecture} is equivalent to 
\begin{equation}\label{dj+ dj- 2}
d_n d_{n-1}<d_0 d_1,\quad\text{for all } n\geq 2.
\end{equation} 
We are going to show \eqref{dj+ dj- 2} by induction. 
For $n=2$, we have to check that $d_1 d_2<d_1 d_0$, which is equivalent to $d_1(d_2-d_0)<0$. Since $d_n\geq 0$ for all $n>1$, it is enough to show that $d_2-d_0<0$.
But, from \eqref{def dj} we easily get that
\begin{equation*}
d_2-d_0=
\frac{ 3 \left(M^3+2 M^2+3 M+3\right) \sinh (M)-3 M \left(M^2+3 M+3\right) \cosh (M)}{e^{M}M^5}
<0.
\end{equation*}

Let us now suppose that \eqref{dj+ dj- 2} holds for $n-1$. 
Then, we can split 
\begin{equation*}
d_{n-1}d_n-d_0d_1=d_{n-1} (d_{n}-d_{n-2})+d_{n-1}d_{n-2}-d_0d_1.
\end{equation*}
On one hand, $d_{n-1}d_{n-2}-d_0 d_1<0$ by \eqref{dj+ dj- 2}. On the other hand, $d_{n}-d_{n-2}\leq 0$ by \cite[Theorem 2]{turan} and $d_{n-1}\geq 0$.
Thus \eqref{dj+ dj- 2} holds for all $n\geq2$.
\end{proof}

\section{Approximation by short - range potentials}
In this section we investigate the spectral relation between the electrostatic $\delta$-shell interaction on the boundary of a smooth domain and its approximation by the coupling of the Dirac operator with a short-range potential which depends on a parameter $\epsilon>0$ in such a way that it shrinks to the boundary of the domain as $\epsilon\to0$; see the definition of $H_{\lambda,\epsilon}$ below. From \cite[Theorem 1.2]{mp} we know that if $a\in\sigma(H_\lambda)$, where here $\sigma(\cdot)$ denotes the spectrum, then there exists a sequence $\seq{a_\epsilon}$ such that $a_\epsilon\in \sigma(H_{\mu,\epsilon})$ and $a_\epsilon\to a$ for $\epsilon\to0$, where  $\lambda=2\tan\left(\frac{\mu}{2}\right)$. However, the vice-versa spectral implication may not hold in general. In this section we are going to show that the reverse does hold in the spherical case, that is, if $a_\epsilon \to a$ with $a_\epsilon\in \sigma(H_{\mu,\epsilon})$, then $a\in\sigma(H_{\lambda})$ (see Theorem \ref{thm: aprox-limit} below). In particular this means that, when passing to the limit, we don't lose any element of the spectrum for electrostatic interactions with potentials shrinking on $\stwo$.

Given $\epsilon>0$ and $x\in\R^3$, we define 
\begin{equation*}
V_\epsilon(x)=\frac{1}{2\epsilon}\chi_{(1-\epsilon,1+\epsilon)}(|x|)\quad
\text{and}\quad\Vep=V_\epsilon\mathbb{I}_4,
\end{equation*} 
where $\mathbb{I}_4$ denotes the $4\times4$ identity matrix.
For $\lambda\in\R$, we also introduce the operators
\begin{align*}
{D}(\mathring H_{\lambda,\epsilon})= C^\infty_c(\Rt\setminus\seq{0})^4\quad&\text{and}\quad \mathring H_{\lambda,\epsilon}=  H+\lambda \Vep,\\
{D}(H_{\lambda,\epsilon})= H^1(\Rt)^4\quad&\text{and}\quad H_{\lambda,\epsilon}=  H+\lambda \Vep.
\end{align*}
Since $|V_\epsilon|\leq\frac{1}{2\epsilon}$,  $\mathring H_{\lambda,\epsilon}$ is essentially self-adjoint and $H_{\lambda,\epsilon}$ is self-adjoint by \cite[Theorem 4.2]{thaller}. Moreover $\sigma_{ess}(H_{\lambda,\epsilon})=\sigma_{ess}(H)=\sigma(H)=(-\infty,-m]\cup[m,+\infty)$. 
For this reason we are looking for some $a\in (-m,m)$ eigenvalue of $H_{\lambda,\epsilon}$. 

Our aim is to find a precise relation between $a$, $\lambda$ and $\epsilon$, say $R_\epsilon(a,\lambda)$, which must hold in order to get an eigenfunction for $H+\lambda V_\epsilon$ with eigenvector $a$. Then, we will take the limit of $R_\epsilon(a,\lambda)$ for $\epsilon\to 0$ and we will compare the result to \eqref{eq eigenvalue}. To do so, we use the same approach developed in \Cref{sec:delta}.

Clearly, if $\lambda=0$ we get that $H_{\lambda,\epsilon}=H$, i.e. we are not perturbing the free Hamiltonian $H$, thus we can exclude this case in our study. Assuming that $\lambda\neq 0$, we note that if $a$ is an eigenvalue of $H_{\lambda,\epsilon}$ with eigenfunction $\psi=\begin{pmatrix}
\phi\\
\chi
\end{pmatrix}$ then $-a$ is an eigenvalue of $H_{-\lambda,\epsilon}$ with eigenfunction  $\tilde{\psi}=\begin{pmatrix}
-\chi\\
\phi
\end{pmatrix}$. For this reason, from now on, we will further assume that $\lambda>0$.

Observe that $\mathring H_{\lambda,\epsilon}$ leaves the partial wave subspace $\mathcal{C}_{m_j,k_j}$ invariant. 
Its action on each subspace is represented with respect to the basis $\seq{\Phi^+_{m_j,k_j},\Phi^-_{m_j,k_J}}$ by the operator
\begin{equation}\label{def H_lambda epsilon}
\begin{split}
&{D}(\mathring h(\lambda,\epsilon)_{m_j,k_j})= C^\infty_c(0,+\infty)^2,
\\
&\mathring 
h(\lambda,\epsilon)_{m_j,k_j}
\begin{pmatrix}
  f \\
  g
\end{pmatrix}
= \left(\begin{array}{cc}
m+\frac{\lambda}{2\epsilon} \chi_{(1-\epsilon , 1+\epsilon )}& -\frac{d}{dr}+\frac{k_{j}}{r}\\
\frac{d}{dr}+\frac{k_{j}}{r} &-m+\frac{\lambda}{2\epsilon} \chi_{(1-\epsilon ,1+\epsilon )} 
\end{array}\right)
\begin{pmatrix}
  f \\
  g
\end{pmatrix}.
\end{split}
\end{equation}
Since $\mathring H_{\lambda,\epsilon}$ is essentially self-adjoint $\mathring h(\lambda,\epsilon)_{m_j,k_j}$ too, thus setting $h(\lambda,\epsilon)_{m_j,k_j}
=\overline{\mathring h(\lambda,\epsilon)_{m_j,k_j}}$ we get
\begin{equation}\label{eq:def.Hleps.H1}
{D}\left(h(\lambda,\epsilon)_{m_j,k_j}\right)
=
\big\{\varphi\in L^2(0,+\infty)^2: h(\lambda,\epsilon)_{m_j,k_j}\varphi\in L^2(0,+\infty)^2 \text{ and  }
\varphi\in AC(0,+\infty)^2\big\}
\end{equation}
and the action of $h(\lambda,\epsilon)_{m_j,k_j}$ on its domain of definition is formally given by the right hand side of the second equation in \eqref{def H_lambda epsilon}.
Moreover, $a\in (-m,m)$ is an eigenvalue for $H_{\lambda,\epsilon}$ if and only if $a$ is an eigenvalue for $h(\lambda,\epsilon)_{m_j,k_j} $ for some $\seq{m_j,k_j}$. For this reason, we want to solve 
\begin{subequations}
\begin{equation*}
\begin{split}
&\begin{cases}
(m-a)f+(-\frac{d}{dr}+\frac{k}{r})g&=0\\
(\frac{d}{dr}+\frac{k}{r})f-(m+a)g&=0
\end{cases}\quad\text{if } 0<r<1-\epsilon\text{ or } r>1+\epsilon,\\
&\begin{cases}
(m-a+\frac{\lambda}{2\epsilon})f+(-\frac{d}{dr}+\frac{k}{r})g&=0\\
(\frac{d}{dr}+\frac{k}{r})f-(m+a-\frac{\lambda}{2\epsilon})g&=0
\end{cases}\quad\text{if } 1-\epsilon<r<1+\epsilon
\end{split}
\end{equation*}
\end{subequations}
for $
\begin{pmatrix}
f\\
g
\end{pmatrix} 
\in {D}\left(h(\lambda,\epsilon)_{m_j,k_j}\right)$. 

Since $k_j=\pm(j+1/2)$, a non-trivial solution is given by 
\begin{equation}\label{eq:def:f,g.epslion.k>0}
\begin{split}
f(r)&=\begin{cases}
\begin{array}{ll}
 A \sqrt{r}\ \II_{\left(j+\frac{1}{2}\right)\pm \frac{1}{2}}(M r) & r<1-\epsilon  \\
 B_1 \sqrt{r}\ \J_{\left(j+\frac{1}{2}\right)\pm \frac{1}{2}}(L r)+B_2 \sqrt{r}\ \Y_{\left(j+\frac{1}{2}\right)\pm \frac{1}{2}}(L r) & 1-\epsilon <r<\epsilon +1 \\
 C \sqrt{r}\ \KK_{\left(j+\frac{1}{2}\right)\pm \frac{1}{2}}(M r) & r>1+\epsilon
\end{array}
\end{cases}\\
g(r)&=\begin{cases}
\begin{array}{ll}
 \frac{A M }{a+m} \sqrt{r}\ \II_{(j+1/2)\mp 1/2}(M r)& 0<r<1-\epsilon, \\
 \frac{L \sqrt{r}}{a-\frac{\lambda}{2 \epsilon}+m} \left(B_1 \ \J_{\left(j+\frac{1}{2}\right)\mp \frac{1}{2}}(L r)+B_2 \  \Y_{\left(j+\frac{1}{2}\right)\mp \frac{1}{2}}(L r)\right) & 1-\epsilon<r<1+\epsilon, \\
 -\frac{C M }{a+m}\sqrt{r}\ \KK_{(j+1/2)\mp 1/2}(M r) & r>1+\epsilon,
\end{array}
\end{cases}
\end{split}
\end{equation}
where $\J$ and $\Y$ denote the first order Bessel's functions and $\II$ and $\KK$ the second order Bessel's functions, 
\begin{equation*}
M=\sqrt{m^2-a^2},\qquad 
L=\sqrt{\left(\frac{\lambda}{2\epsilon}-a\right)^2-m^2}
\end{equation*}
and $(A,B_1,B_2,C)\neq 0$ are some constants. Note that $M\in\R$ by the assumptions on $a$, but $L$ could be complex.
Note also that $f,g\in H^1\big((0,+\infty)\setminus\seq{1-\epsilon,1+\epsilon},r\,dr\big)$. To ensure that they belong to ${D}\big(h(\lambda,\epsilon)_{m_j,k}\big)$ we have to verify that both $f$ and $g$ are continuous in $1-\epsilon$ and $1+\epsilon$, which means that the following linear system must hold:
\begin{equation}\label{eq:system}
\begin{cases}
\begin{array}{cl}
0&= 
A \sqrt{1-\epsilon }\ \II_{(j+1/2)\pm 1/2}(M(1- \epsilon) )
-B_1 \sqrt{1-\epsilon}\ {\J}_{(j+1/2)\pm1/2}(L(1- \epsilon)) \\
&\quad
- B_2 \sqrt{1-\epsilon}\ \Y_{(j+1/2)\pm1/2}(L(1- \epsilon) ),\\ \\
0&= 
A\dfrac{  \sqrt{1-\epsilon }\ M\  \II_{(j+1/2)\mp 1/2}(M(1- \epsilon) )}{a+m}
- B_1\dfrac{2 \epsilon L \sqrt{1-\epsilon}\ {\J}_{(j+1/2)\mp 1/2}(L(1- \epsilon ))}{2 a \epsilon -\lambda +2 m \epsilon }\\ 
&\quad
-B_2 \dfrac{2 \epsilon L  \sqrt{1-\epsilon}\  \Y_{(j+1/2)\mp 1/2}(L(1- \epsilon ))}{2 a \epsilon -\lambda +2 m \epsilon },\\ \\
0&=  
B_1 \sqrt{1+\epsilon}\ {\J}_{(j+1/2)\pm1/2}(L(1+\epsilon))
+B_2 \sqrt{1+\epsilon}\ {\Y}_{(j+1/2)\pm1/2}(L(1+\epsilon))\\
&\quad-
C \sqrt{1+\epsilon}\ \KK_{(j+1/2)\pm 1/2}(M(1+\epsilon)),\\  \\
0&= 
B_1 \dfrac{2\epsilon L  \sqrt{1+\epsilon}\ {\J}_{(j+1/2)\mp 1/2}(L(1+\epsilon))}{2 a \epsilon -\lambda +2 m \epsilon }+B_2\dfrac{2\epsilon  L \sqrt{1+\epsilon}\  {\Y}_{(j+1/2)\mp 1/2}(L(1+\epsilon))}{2 a \epsilon -\lambda +2 m \epsilon }\\ 
&\quad+C\dfrac{  \sqrt{1+\epsilon}\ M\ \KK_{(j+1/2)\mp 1/2}(M(1+\epsilon))}{a+m}.
\end{array}
\end{cases}
 \end{equation}
Since this is a $4\times 4$ homogeneous linear system on $A$, $B_1$, $B_2$ and $C$ and we are assuming that $(A,B_1,B_2,C)\neq 0$, we deduce that the associated matrix has null determinant. 
So, if we set
\begin{equation}\label{condition epsilon}
\begin{split}
&D^\epsilon_{j\pm1/2}(a,\lambda):=\\
&\hskip18pt\frac{2   (a+m)\ \KK_{(j+1/2)\pm1/2}\big(\sqrt{m^2-a^2}(1+\epsilon)\big)}{\epsilon (-2 a \epsilon +\lambda -2 m \epsilon )^2}\\
&\hskip30pt\times\Bigl\lbrace -2 L \epsilon  (a+m)\ \II_{(j+1/2)\pm1/2}\big(\sqrt{m^2-a^2}(1-\epsilon) \big)\\
&\hskip68pt
\times
\bigl[\ \J_{(j+1/2)\mp1/2}(L(1+\epsilon))\ \Y_{(j+1/2)\mp1/2}(L(1-\epsilon) )\\
&\hskip84pt- \J_{(j+1/2)\mp1/2}(L(1-\epsilon) )\ \Y_{(j+1/2)\mp1/2}(L(1+\epsilon))\bigr]\\
&\hskip48pt-\sqrt{m^2-a^2} (2 a \epsilon -\lambda +2 m \epsilon )\ \II_{(j+1/2)\mp 1/2}\big(\sqrt{m^2-a^2}(1-\epsilon) \big)\\ 
&\hskip68pt\times\bigl[\ \J_{(j+1/2)\pm1/2}(L(1-\epsilon) )\ \Y_{(j+1/2)\mp1/2}(L(1+\epsilon))\\
&\hskip84pt
-\J_{(j+1/2)\mp1/2}(L(1+\epsilon))\ \Y_{(j+1/2)\pm1/2}(L(1-\epsilon))\bigr]\Bigr\rbrace\\
&\hskip18pt+\frac{\sqrt{m^2-a^2}\ \KK_{(j+1/2)\mp 1/2}\big(\sqrt{m^2-a^2}(1+\epsilon)\big)}{\epsilon^2 L(2 a \epsilon -\lambda +2 m \epsilon) }\\ 
&\hskip30pt\times\Bigl\lbrace -2 L \epsilon  (a+m)\ \II_{(j+1/2)\pm1/2)}\big(\sqrt{m^2-a^2}(1-\epsilon) \big)\\ 
&\hskip68pt\times
\bigl[\ \J_{(j+1/2)\pm1/2}(L(1+\epsilon))\ \Y_{(j+1/2)\mp1/2}(L(1-\epsilon) )\\
&\hskip84pt-\J_{(j+1/2)\mp1/2}(L(1-\epsilon) )\ \Y_{(j+1/2)\pm1/2}(L(1+\epsilon))\bigr]\\
&\hskip48pt+\sqrt{m^2-a^2} (2 a \epsilon -\lambda +2 m \epsilon )\ \II_{(j+1/2)\mp1/2}\big(\sqrt{m^2-a^2}(1-\epsilon) \big) \\
&\hskip68pt\times\bigl[\ \J_{(j+1/2)\pm1/2}(L(1+\epsilon))\ \Y_{(j+1/2)\pm1/2}(L(1-\epsilon) )\\
&\hskip84pt-\J_{(j+1/2)\pm1/2}(L(1-\epsilon) ) \Y_{(j+1/2)\pm1/2}(L(1+\epsilon))\bigr]\Bigr\rbrace,
\end{split}
\end{equation}
then the determinant of the matrix associated to the linear system \eqref{eq:system} is 
$\frac{\epsilon(\epsilon^2-1)}{(a+m)^2}D^\epsilon_{j\pm1/2}(a,\lambda)$, which vanishes if and only if 
\begin{equation}\label{condition epsilon'}
D^\epsilon_{j\pm1/2}(a,\lambda)=0.
\end{equation}  
We can conclude that, given 
\begin{equation*}
\varphi=\begin{pmatrix}
f\\
g
\end{pmatrix}\quad\text{with $f$ and $g$ as in \eqref{eq:def:f,g.epslion.k>0},}
\end{equation*}
$\varphi$ is an eigenfunction for $h(\lambda,\epsilon)_{m_j,k_j}$ with eigenvalue $a$ if and only if $D^\epsilon_{j\pm1/2}(a,\lambda)=0$. This means that the function 
\begin{equation*}
\psi(x)=
\frac{1}{r}
\left(
f(r)\Phi^+_{m_j,k_j}(\hx)
+g(r)\Phi^-_{m_j,k_j}(\hx)
\right)
\end{equation*} 
is an eigenfunction for $H_{\lambda,\epsilon}$ with eigenvalue $a$. 

In order to compare \eqref{condition epsilon} and \eqref{eq:final.eigenvalue}, let us draw some pictures of these relations for some concrete values of the underlying parameters, say $m=1$, $k=1$ and $\epsilon=2^{-10}$. Figures \ref{fid:limit} and \ref{fig:delta} describe  the set of 
$(a,\lambda)\in (-1,1)\times(0,10)$ that verify \eqref{condition epsilon'} and \eqref{eq:final.eigenvalue}, respectively.
\noindent
\begin{minipage}{\linewidth}
      \centering
      \begin{minipage}{0.45\linewidth}
          \begin{figure}[H]
              \includegraphics[width=\linewidth]{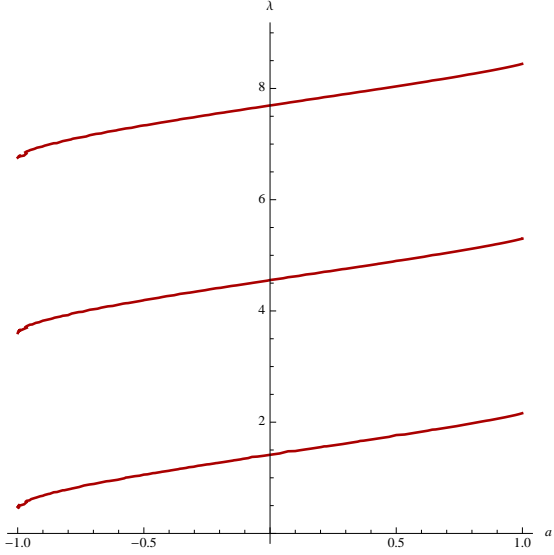}
              \caption{The set of points $(a,\lambda)$ satisfying \eqref{condition epsilon'}.}
              \label{fid:limit}
          \end{figure}
      \end{minipage}
      \hspace{0.05\linewidth}
      \begin{minipage}{0.45\linewidth}
          \begin{figure}[H]
              \includegraphics[width=\linewidth]{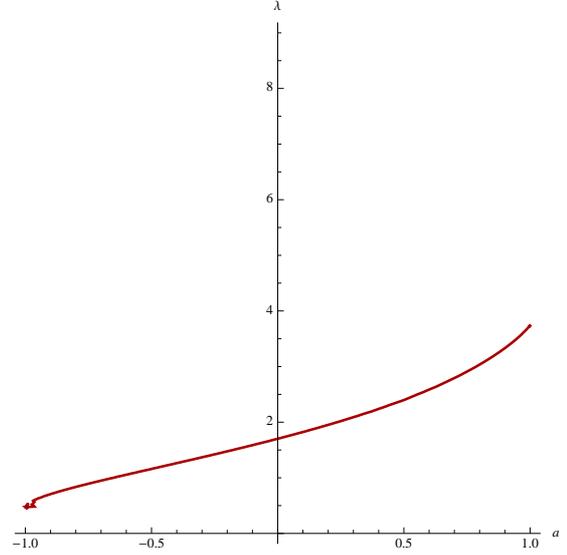}
              \caption{The set of points $(a,\lambda)$ satisfying \eqref{eq:final.eigenvalue}.}
              \label{fig:delta}
          \end{figure}
      \end{minipage}
  \end{minipage}

Looking at Figures \ref{fid:limit} and \ref{fig:delta} we note that there is no apparent relation between \eqref{condition epsilon'} and \eqref{eq:final.eigenvalue}. However, the next result proves that there is indeed a precise connection between both equations when one takes the limit $\epsilon\to0$ in $D^\epsilon_{j\pm1/2}(a,\lambda)$.
\begin{lemma}\label{lem:d_e->d}
Let $j=1/2,\,3/2,\dots$ and $D^\epsilon_{j\pm1/2}$ and $D_{j\pm1/2}$ be defined  by \eqref{condition epsilon} and \eqref{condition limit}, respectively. Then, for any $\lambda>0$,
\begin{equation*}
\lim_{\epsilon\to 0}
D^\epsilon_{j\pm1/2}(a,\lambda)=
 \frac{4(a+m)}{\lambda\pi\left(1+\tan\left(\frac{\lambda}{2}\right)^2\right)} D_{j\pm1/2}\left(a,2\tan\left(\textstyle{\frac{\lambda}{2}}\right)\right)\
\quad
\text{uniformly on }a\in (-m,m).
\end{equation*}
\end{lemma}
\begin{proof}
Note that $L\to +\infty$ uniformly in $a\in(-m,m)$ when $\epsilon\to 0$, thus we can use the asymptotics 
\begin{align*}
\J_n(z)&=\sqrt{\frac{2}{\pi z}}\left(\cos\left(z-\textstyle{\frac{1}{2}}n\pi-\textstyle{\frac{1}{4}}\pi\right)+e^{|\Im(z)|}o(1)\right)
\quad\text{for $|z|\to+\infty$},\\
\Y_n(z)&= \sqrt{\frac{2}{\pi z}}\left(\sin\left(z-\textstyle{\frac{1}{2}}n\pi-\textstyle{\frac{1}{4}}\pi\right)+e^{|\Im(z)|}o(1)\right)
\quad\text{for $|z|\to+\infty$},
\end{align*}
see \cite[Equation 10.7.8]{olver2010nist}.
Inserting these two relations in \eqref{condition epsilon} and taking $\epsilon\to 0$, we get that, uniformly on $a\in(-m,m)$,
\begin{equation*}
\begin{split}
&\lim_{\epsilon\to 0 }D^\epsilon_{j\pm 1/2}(a,\lambda)\\
&\!\quad= \frac{4}{\lambda\pi}\Bigl(M\, \II_{(j+1/2)\mp 1/2}(M) \left((a+m) \cos (\lambda) \ \KK_{(j+1/2)\pm 1/2}(M)-M \sin (\lambda)\ \KK_{(j+1/2)\mp 1/2}(M)\right)\\
&\!\qquad+(a+m)\ \II_{(j+1/2)\pm 1/2}(M) \left((a+m) \sin (\lambda)\ \KK_{(j+1/2)\pm 1/2}(M)+M \cos (\lambda)\ \KK_{(j+1/2)\mp 1/2}(M)\right)\Bigr).
\end{split}
\end{equation*}
Setting $t=2\tan\left(\frac{\lambda}{2}\right)$, we know that $\sin(\lambda)=\frac{t}{1+\frac{t^2}{4}}$ and $\cos(\lambda)=\frac{1-\frac{t^2}{4}}{1+\frac{t^2}{4}}$. Using \eqref{eq:wronskian}, hence\begin{equation*}
\begin{split}
\lim_{\epsilon\to 0 }D^\epsilon_{j\pm 1/2}(a,\lambda)=\frac{16(a+m)}{\lambda\pi(4+t^2)}\Bigl(\frac{t^2}{4}-\Big((m+a)&\ \II_{(j+1/2)\pm 1/2}(M)\ \KK_{(j+1/2)\pm 1/2}(M)\\
&-(m-a)\ \II_{(j+1/2)\mp 1/2}(M)\ \KK_{(j+1/2)\mp 1/2}(M)\Big)t-1\Bigr),
\end{split}
\end{equation*}
which coincides with \eqref{condition limit} if one replaces $\lambda$ by $t=2\tan\left(\frac{\lambda}{2}\right)$ in there.
\end{proof}

The following result resumes what we have proven so far with the aid of Lemma \ref{lem:d_e->d}.
\begin{theorem}\label{thm: aprox-limit}
Let $\mu\in\R\setminus\{0\}$ and
\begin{equation*}
\lambda=2\tan\left(\frac{\mu}{2}\right).
\end{equation*}
Let $h(\lambda)_{m_j,k_j}$ be as in \eqref{def: h(lambda)mj,kj} and, for 
$\epsilon>0$, let $h(\mu,\epsilon)_{m_j,k_j}$ be as in \eqref{eq:def.Hleps.H1}. If 
$a_\epsilon\in\sigma_{p}(h(\mu,\epsilon)_{m_j,k_j})$ and
$\lim_{\epsilon\to 0} a_\epsilon =a$ for some $a\in(-m,m)$, 
then $a\in \sigma_p(h(\lambda)_{m_j,k_j})$.
\end{theorem}
\section{Acknowledgements}
We would like to thank Aingeru Fern\'andez-Bertol\'in for adressing us the reference \cite{turan} and for the enlightening discussions. 
Mas was supported by the {\em Juan de la Cierva} program JCI2012-14073 and the funding projects MTM2011-27739 and MTM2014-52402 from the Government of Spain. Pizzichillo was supported by the MINECO project MTM2014-53145-P, by the Basque Government through the BERC 2014-2017 program and by the Spanish Ministry of Economy and Competitiveness MINECO: BCAM Severo Ochoa accreditation SEV-2013-0323. Both authors were also supported by the ERCEA, Advanced Grant project 669689 HADE.

\end{document}